\documentclass[a4paper,12pt]{article}
\usepackage{a4wide}
\usepackage{amsmath}
\usepackage{amssymb}
\usepackage{amsthm}
\usepackage{latexsym}
\usepackage{graphicx}
\usepackage[english]{babel}
\usepackage{makeidx}
\usepackage{mathtools}
              
\newtheorem{dfn} [subsection]{Definition}
\newtheorem{obs} [subsection]{Remark}
\newtheorem{exm} [subsection]{Example}

\newtheorem{prop}[subsection]{Proposition}
\newtheorem{conj}[subsection]{Conjecture}
\newtheorem{teor}[subsection]{Theorem}

\newtheorem{cor} [subsection]{Corollary}

\def\supp{\operatorname{Cons}}
\def\Cons{\operatorname{Cons}}

\def\Ch{\operatorname{Ch}}
\def\Char{\operatorname{Char}}

\def\Lin{\operatorname{Lin}}
\def\Hilb{\operatorname{Hilb}}

\def\Irr{\operatorname{Irr}}

\def\Hol{\operatorname{Hol}}

\def\Sup{\operatorname{Sup}}
\def\Reg{\operatorname{Reg}}
\def\Mon{\operatorname{Mon}}

\def\Cons{\operatorname{Cons}}
\def\supp{\operatorname{supp}}
\def\Ker{\operatorname{Ker}}
\numberwithin{equation}{section}
\usepackage{titletoc}
\dottedcontents{section}[2.5em]{}{2em}{1pc}
\begin{document}
\selectlanguage{english}
\frenchspacing

\large
\begin{center}
\textbf{On the monomial algebra associated to the monomial characters of a finite group}

Mircea Cimpoea\c s and Alexandru F. Radu
\end{center}
\normalsize

\begin{abstract} 

Given a finite group $G$, we study the monomial algebra $R_G$, generated by the monomial characters of $G$.
In particular, we note that the integral closure of $R_G$ is contained in the algebra generated by those
characters $\chi$ for which their associated Artin L-function $L(s,\chi)$ is holomorphic 
at $s_0\in\mathbb C\setminus\{1\}$. Also, we discuss the supercharacter theoretic case.


\end{abstract}

\section{Introduction}

A group $G$ is called \emph{monomial}, if every complex irreducible character $\chi$ of $G$ is induced by a linear character $\lambda$ of
a subgroup $H$ of $G$, that is $\chi=\lambda^G$. A group $G$ is called \emph{quasi-monomial}, if for every irreducible character $\chi$ of $G$, 
there exists a subgroup $H$ of $G$ and a linear character $\lambda$ of $H$ such that $\lambda^G=d\chi$, where $d$ is a positive integer. 
A finite group $G$ is called {\em almost monomial} if for every distinct complex irreducible characters $\chi$ and $\psi$ of  $G$ 
there exist a subgroup $H$ of $G$ and a linear character $\lambda$ of $H$ such that the induced character $\lambda^G$ contains 
$\chi$ and does not contain  $\psi$. 

Let $G$ be a finite group with $\Irr(G)=\{\chi_1,\ldots,\chi_r\}$. Let $M(G)$ be the monoid generated by the monomial characters of $G$.
Let $K$ be an arbitrary field and let $S=K[x_1,\ldots,x_r]$ be the polynomial ring in $n$ indeterminates over $K$. 
We consider the $K$-algebra $R_G=K[M(G)]$ and we assume that it is minimally generated by the monomials $u_1,\ldots,u_p$.
In Proposition $2.5$ we give a characterization of almost monomial groups in terms of $u_1,\ldots,u_p$.
Let $N\unlhd G$ be a normal subgroup pf $G$. In Proposition $2.9$,
we prove that $R_{G/N}=R_G\cap K[x_1,\ldots,x_s]$, where $N\subseteq \Ker(\chi_i)$ if and only if $1\leq i\leq s$.
In Theorem $2.10$, we prove that $R_{G_1\times G_2}=R_{G_1}\otimes_K R_{G_2}$.

Let $K/\mathbb Q$ be a finite Galois extension with Galois group $G$. For any character $\chi$ of $G$, let $L(s,\chi):=L(s,\chi,K/\mathbb Q)$
be the corresponding Artin L-function (\cite[P.296]{artin2}). Let $\Hol(s_0)$ be the semigroup of Artin $L$-functions, holomorphic at $s_0\in\mathbb C\setminus\{1\}$;
see \cite{monat} for further details. In Theorem $2.11$, we prove that the integral closure of $R_G$ is a subalgebra of $K[\Hol(s_0)]$.

The notion of a supercharacter theory for a finite group was introduced in 2008, by Diaconis and Isaacs \cite{isaacs}, as follows: 
A supercharacter theory of a finite group $G$, is a pair $C=(\mathcal X,\mathcal K)$ where $\mathcal X=\{X_1,\ldots,X_r\}$
is a partition of $\Irr(G)$, the set of irreducible characters of $G$, and $\mathcal K$ is a partition of $G$, such that: (1) $\{1\}\in\mathcal K$, 
(2) $|\mathcal X|=|\mathcal K|$ and (3) $\sigma_X:=\sum_{\psi\in X}\psi(1)\psi$ is constant for each $X\in\mathcal X$ and $K\in\mathcal K$. 
In Section $3$, we extend our main results in the frame of supercharacter theory.
In Section $4$, we presend a GAP code \cite{gap} which returns the Hilbert basis of the monoid $M(G)$. Also, we conjecture a formula
for $R_G$, where $G=SL(2,2^n)$; see Conjecture $4.1$.

\newpage
\section{Main results}

Let $G$ be a finite group with $\Irr(G)=\{\chi_1,\ldots,\chi_r\}$, the set of irreducible (complex) characters.
It is well known that any character $\chi$ of $G$ can be written uniquely as $$\chi=a_1\chi_1+\cdots+a_r\chi_r,$$
where $a_i$'s are non-negative integers and at least on of them is positive. 

A \emph{virtual character} $\chi$ of $G$ is a linear
combination $\chi:=a_1\chi_1+\cdots+a_r\chi_r$, with $a_i\in\mathbb Z$. The set $\Char(G)$ of virtual characters of $G$ has
a structure of a ring and it's called the is the character ring of $G$. In particular, $(\Char(G),+)$ is an abelian group,
isomorphic to $(\mathbb Z^r,+)$. We denote $\Ch_+(G)$, the set of characters of $G$. $\Ch(G):=\Ch^+(G)\cup\{0\}$ is a 
submonoid of $(\Char(G),+)$, isomorphic to $(\mathbb N^r,+)$.

A character $\chi$ is \emph{monomial}, if there exists a subgroup $H\leqslant G$ and a linear character $\lambda$ of $H$ such that $\lambda^G=\chi$.
We denote $\Mon(G)$, the set of monomial characters of $G$. We let $M(G)$ be the submonoid of $\Ch(G)$, generated by $\Mon(G)$, that is:
$$M(G)=\{\chi\in \Ch(G)\;:\;\chi=\psi_1+\cdots+\psi_m\text{ where }m\geq 0,\;\psi_i\in\Mon(G),\text{ for all }1\leq i\leq m\}.$$
Note that $M(G)$ is a finitely generated submonoid of $\Ch(G)$, hence it is affine. Moreover, $M(G)$ is positive, as $0$ is the only unit of it.
Let $\Hilb(M(G))\subset \Mon(G)$ be the Hilbert basis of $M(G)$, that is the minimal system of generators.

We assume that $\Hilb(M(G))=\{\sigma_1,\ldots,\sigma_p\}$, where $\sigma_i=a_{i1}\chi_1+\cdots+a_{ir}\chi_r,$
for some nonnegative integers $a_{ij}$. We denote $A:=A(G)=(a_{ij})_{\substack{i=\overline{1,p} \\ j=\overline{1,r}}} \in \mathcal M_{p,r}(\mathbb Z)$.

Brauer's induction theorem \cite{brauer} shows that $(\Char(G),+)$ is generated, as a group, by $\Mon(G)$. In particular, $\Hilb(M(G))$ generates $(\Char(G),+)$ as 
a group. In particular, $M(G)$ is a monoid of the rank $r$ and therefore $p\geq r$.

Let $K$ be an arbitrary field and let $S=K[x_1,\ldots,x_r]$ be the polynomial ring in $n$ indeterminates over $K$.
We consider the $K$-algebra
$$ R_G:=K[M(G)]:=K[u_1,\ldots,u_p] \subset S,\;u_i:=\prod_{j=1}^r x_j^{a_{ij}},\;1\leq i\leq p.$$

\begin{obs}\emph{
Assuming that $\chi_1=1_G$ is the trivial character of $G$, it follows that $u_1=x_1 \in R_G$. Let $\Reg(G)=\sum_{i=1}^r d_i\chi_i$
be the regular character of $G$ and let 
$$u_G:=x_1^{d_1}x_2^{d_2}\cdots x_r^{d_r}\in R_G$$ be the corresponding monomial.
According do Brauer-Aramata Theorem, see \cite{arama} and \cite{brauer}, $\Reg(G)-1_G$ can be written as a linear combination with rational
positive coefficients of monomial characters. It follows that there exists an integer $\alpha \geq 1$ such that $\alpha (\Reg(G)-1_G)\in M(G)$.
Therefore: 
\begin{equation}\label{unu}
(u_G/x_1)^{\alpha} = x_2^{\alpha d_2}\cdots x_r^{\alpha d_r}\in R_G.
\end{equation}
Moreover, according to \cite[Theorem 2]{rhoades}, which generalizes Aramata-Brauer Theorem, using a similar argument as above, for any $j\in\{1,\ldots,r\}$,
there exists a smallest integer $\alpha_j\geq 1$ such that: 
\begin{equation}\label{doi}
(u_G/x_j)^{\alpha_j} = x_1^{\alpha_j d_1}\cdots x_{j-1}^{\alpha_j d_{j-1}}\cdot x_j^{\alpha_j(d_{j}-1)}\cdot
   x_{j+1}^{\alpha_j d_{j+1}} \cdots x_{r}^{\alpha_j d_{r}} \in R_G.
\end{equation}}
\end{obs}

We have the canonical epimorphism: $$\Phi:K[t_1,\ldots,t_p]\to R_G,\;\Phi(t_i):=u_i,\;1\leq i\leq p.$$
The ideal $I_G:=\Ker(\Phi)$ is the toric ideal of $M(G)$. Note that $I_G=\{0\}$, that is, the monoid $M(G)$ is factorial, 
if and only if $r=p$.

According to Brauer's induction Theorem, any irreducible character $\chi$ of $G$, can be written as a linear combination of monomial characters,
with integer coefficients. It follows that:
\begin{equation}\label{cur}
K[u_1^{\pm 1},\ldots,u_p^{\pm 1}] = K[x_1^{\pm 1},\ldots,x_r^{\pm 1}],
\end{equation}
the ring of Laurent polynomials. In particular, we have that $\dim(M(G))=r$, hence $p\geq r$. Moreover, 
$p=r$ if and only if $I_G=\{0\}$.

The identity \eqref{cur} is equivalent to say that there exists a $B\in \mathcal M_{r,p}(\mathbb Z)$ such that
$A\cdot B = I_r$. In particular, if $p=r$, this is equivalent to $\det(A)\in \{-1,+1\}$.

We state without proof, the following (obvious) result:

\begin{prop}\label{p21}
Let $G$ be a finite group. Then:
\begin{enumerate}
\item[(1)] $G$ is monomial if and only if $\Hilb(M(G))=\Irr(G)$ if and only if $R_G=S=K[x_1,\ldots,x_r]$.
\item[(2)] $G$ is quasi-monomial if and only if $\Hilb(M(G))=\{a_1\chi_1,\ldots,a_r\chi_r\}$ if and only if $R_G=S=K[x_1^{a_1},\ldots,x_r^{a_r}]$,
           where $a_1,\ldots,a_r$ are some positive integers.
\end{enumerate}
\end{prop}

According to Proposition \ref{p21}, if $G$ is (quasi)-monomial, then $I_G=\{0\}$. 
We propose the following Conjecture:

\begin{conj}\label{conj1}
If $G$ is a finite group with $I_G=\{0\}$ then $G$ is monomial.
\end{conj}

For a character $\psi$ of $G$, we denote $\Cons(\psi)$ the set of constituents of $\psi$.
We recall the following definition, which generalize quasi-monomial groups:

\begin{dfn}(\cite{monat})\label{am-def}
A finite group $G$ is called {\em almost monomial} (or AM-group) if for every two distinct characters $\chi \neq \psi \in \Irr(G)$, there exist a 
subgroup $H$ of $G$ and a linear character $\lambda$ of $H$ such $\chi\in \Cons(\lambda^G)$ and $\psi\notin \Cons(\lambda^G)$.
\end{dfn}

Regarding Conjecture \ref{conj1}, we are able to prove the following result:

\begin{prop}\label{patru}
Let $G$ be an almost monomial group with $r=|\Irr(G)|\leq 4$ and $I_G=(0)$. Then $G$ is monomial.
\end{prop}

\begin{proof}
Assume that $R_G=K[u_1,\ldots,u_r]$ and $u_i=\prod_{j=1}^r x_j^{a_{ij}}$, $1\leq i\leq r$.
Let $A=(a_{ij})_{\substack{i=\overline{1,r} \\ j=\overline{1,r}}}$. We note that $G$ is monomial
if and only if $A$ is a permutation matrix. We assume that $\chi_1$ is the trivial character of $G$, 
and thus $a_{11}=1$ and $a_{1j}=0$ for $2\leq j\leq r$.
\begin{enumerate}
\item[(i)] $r=1$. This case is trivial.
\item[(ii)] $r=2$. We have $A=\begin{pmatrix} 1 & a_{12} \\ 0 & a_{22}  \end{pmatrix}$. 
     For $2\neq 1$, there exists $i\in\{1,2\}$, such that $x_2|u_i$ and $x_1\nmid u_i$. 
     Since $u_1=x_1$, it follows that $x_1\nmid u_2$, hence $a_{12}=0$ and $a_{22}>0$. Since $\det(A)=\pm 1$ and $a_{22}>0$, it follows that $a_{22}=1$.
     Therefore $A=I_2=\begin{pmatrix} 1 & 0 \\ 0 & 1  \end{pmatrix}$ is a permutation matrix. 
\item[(iii)] $r=3$. We have $A=\begin{pmatrix} 1 & a_{12} & a_{13} \\ 0 & a_{22} & a_{23} \\ 0 & a_{32} & a_{33} \end{pmatrix}$.
      Up to a renumerotation, we may assume that $x_2|u_2$, $x_3\nmid u_2$, $x_3|u_3$ and $x_2\nmid u_3$. It follows that
      $a_{31}=a_{23}=0$ and $\det(A)=\begin{vmatrix} 1 & a_{12} & a_{13} \\ 0 & a_{22} & 0 \\ 0 & 0 & a_{33} \end{vmatrix} = a_{22}a_{33}=1$,
      hence $a_{22}=a_{33}=1$. So $u_1=x_1$, $u_2=x_1^{a_{12}}x_2$, $u_3=x_1^{a_{13}}x_3$.
      If $a_{12}>0$, then there is no $1\leq i\leq 3$ such that $x_2|u_1$ and $x_1\nmid u_1$, a contradiction. Similarly, the assumption $a_{13}>0$
      leads to a contradiction. It follows that $A=I_3$.
\item[(iv)] $r=4$. Up to a renumerotation, we may assume that $x_2|u_2$, $x_3\nmid u_2$, $x_3|u_3$ and $x_2\nmid u_3$. This implies $a_{22}>0$, $a_{23}=0$, $a_{32}=0$ and $a_{33}>0$.
     Since $x_2\nmid u_1$ and $x_2\nmid u_3$, it follows that
     either (a)  $x_2|u_4$ and $x_4\nmid u_4$ or (b) $x_4\nmid u_2$.
\begin{enumerate}
 \item[(a)] We have that \small
     \begin{equation}\label{curv}
     \det(A)=\begin{vmatrix} 1 & a_{12} & a_{13} & a_{14} \\ 0 & a_{22} & 0 & a_{24} \\ 0 & 0 & a_{33} & a_{34} \\ 0 & a_{42} & a_{43} & 0 \end{vmatrix}=
     -a_{24}a_{33}a_{42}-a_{34}a_{43}a_{22}=\pm 1,\end{equation}
     \normalsize
     If $a_{42}=0$, then there is no $u_i$ such that $x_4|u_i$ and $x_2\nmid u_i$, a contradiction, hence $a_{42}\neq 0$. Similarly, we have $a_{43}\neq 0$.
     Since $a_{22}>0$ and $a_{42}>0$, we must have $a_{24}>0$, otherwise there is no $u_i$ with $x_2|u_i$ and $x_4\nmid u_i$. From \eqref{curv} it follows that $a_{24}=a_{33}=a_{42}=1$
     and $a_{34}=0$, hence $A=\begin{pmatrix} 1 & a_{12} & a_{13} & a_{14} \\ 0 & a_{22} & 0 & 1 \\ 0 & 0 & 1 & 0 \\ 0 & a_{42} & a_{43} & 0 \end{pmatrix}$. Since $a_{43}>0$, there is
     no $u_i$ such that $x_3|u_i$ and $x_4\nmid u_i$, a contradiction.

\item[(b)] Since $a_{42}=0$, we have that  $A=\begin{pmatrix} 1 & a_{12} & a_{13} & a_{14} \\ 0 & a_{22} & 0 & a_{24} \\ 0 & 0 & a_{33} & a_{34} \\ 0 & 0 & a_{43} & a_{44} \end{pmatrix}$.
     Note that $u_3$ is the only possible choice for $x_3|u_3$ and $x_4\nmid u_4$, therefore $a_{43}=0$. Since $x_4\nmid u_i$ for $i\neq 3$, it follows that $x_4|u_4$ and $x_3\nmid u_4$,
     therefore $a_{34}=0$. It follows that $\det(A)=a_{22}a_{33}a_{44}=\pm 1$, hence $a_{22}=a_{33}=a_{44}=1$. We have  
     $A=\begin{pmatrix} 1 & a_{12} & a_{13} & a_{14} \\ 0 & 1 & 0 & a_{24} \\ 0 & 0 & 1 & a_{34} \\ 0 & 0 & 0 & 1 \end{pmatrix}$. Now, since for any $4\neq j$ there exists $u_i$ such
     that $x_4|u_i$ and $x_j\nmid u_i$, it follows that $a_{14}=a_{24}=a_{34}=0$. Similarly, $a_{12}=a_{13}=0$. Thus $A=I_4$.
\end{enumerate}
\end{enumerate}
\end{proof}

\begin{obs}
\emph{
The conclusion of Proposition \ref{patru} can be also obtained from the classification of finite groups according
to the number of conjugacy classes, $r=r(G)\leq 4$, see \cite[TABLE 1]{cc}.  
The method used in the proof of the Proposition \ref{patru} does not work for $r\geq 5$. For instance, 
the matrix $A=\begin{pmatrix} 1&0&0&0&0 \\ 0&2&0&1&0\\ 0&0&1&0&1\\ 0&1&1&0&0\\0&0&0&1&1 \end{pmatrix}$ has
the determinant $\det(A)=1$. Also, if we let $u_1=x_1$, $u_2=x_2^2x_4$, $u_3=x_3x_4$, $u_4=x_2x_5$ 
and $u_5=x_3x_5$, then it is easy to check that for any $j\neq k$, there exist $i$ such that $x_j|u_i$
and $x_k\nmid u_i$.}
\end{obs}

For a monomial $u\in K[x_1,\ldots,x_r]$, we denote $\supp(u)=\{x_i\;:\;x_i|u\}$, the \emph{support} of $u$.

\begin{prop}\label{25}
Let $G$ be a finite group with $\Irr(G)=\{\chi_1,\ldots,\chi_r\}$ and $R_G=K[u_1,\ldots,u_p]$. The following are equivalent:
\begin{enumerate}
 \item[(1)] $G$ is almost monomial.
 \item[(2)] For any $j\neq k \in \{1,\ldots,r\}$, there exists $i\in\{1,\ldots,m\}$ such that $x_j| u_i$ and $x_k\nmid u_i$.
 \item[(3)] For any $k\in \{1,\ldots,r\}$, there exists a subset $I_j\subset \{1,2,\ldots,r\}$ such that 
       $$\supp(\prod_{i\in I_j}u_i) = \{x_1,\ldots,x_{k-1},x_{k+1},\ldots,x_r\}.$$
\end{enumerate}
\end{prop}

\begin{proof}
$(1)\Rightarrow (2)$ By Definition \ref{am-def}, there exists a subgroup $H\leqslant G$ and a linear character $\lambda$ of $H$ such
that $\chi_j \in \Cons(\lambda^G)$ and $\chi_k\notin \Cons(\lambda^G)$. Let $u:=\prod_{i=1}^r x_i^{\langle \lambda^G,\chi_i \rangle}$.
It follows that $x_j|u$ and $x_k\nmid u$. Since $u\in R_G$, there exists $i\in\{1,\ldots,m\}$ such that $u_i|u$ and $x_j|u_i$.
Since $x_k\nmid u$, it follows that $x_k\nmid u_i$.

$(2)\Rightarrow (3)$ For each $j\in \{1,\ldots,r\}\setminus \{k\}$, let $i_j\in\{1,\ldots,m\}$ such that $x_j|u_{i_j}$ and $x_k\nmid u_{i_j}$.
Let $I_j:=\{i_j\;:\; j\in \{1,\ldots,r\}\setminus \{k\}\}$. Then $\supp(\prod_{i\in I_j}u_i) = \{x_1,\ldots,x_{k-1},x_{k+1},\ldots,x_r\}$,
as required.

$(3)\Rightarrow (1)$ Follows immediately from \cite[Proposition 1.3]{cimrad}.
\end{proof}

\begin{exm}\emph{
(1) Let $G:=SL(2,3)$. Then $R_G=K[x_1,x_2,x_3, x_7, x_4x_5, x_4x_6, x_5x_6, x_4x_5x_6]$. Note that $G$ satisfies the condition (2) of Proposition \ref{25}, hence
it is almost monomial. Moreover, we have $R_G\cong K[t_1,\ldots,t_8]/(t_5t_6t_7-t_8^2)$.}

\emph{
(2) Let $G:=GL(2,3)$. Then $R_G=K[x_1,x_2,x_3, x_6,x_7,x_8, x_4x_8, x_5x_8, x_4x_5x_8]$. Note that there is no monomial $u\in R_G$
with $\supp(u)=\{x_1,\ldots,x_7\}$. Hence, by Proposition \ref{25}(3), it follows that $G$ is not almost monomial. 
Moreover, we have $R_G\cong K[t_1,\ldots,t_9]/(t_6t_9-t_7t_8)$.}
\end{exm}

Let $G$ be a finite group. We denote by $\widehat{M(G)}$, the integral closure of $M(G)$ in $\Ch(G)$.
Also, we consider the (normal) monomial algebra $\widehat{R_G}:=K[\widehat{M(G)}]$. Obviously, $R_G\subset \widehat{R_G}$.

\begin{prop}
For any finite group $G$, we have:
$$K[x_1,\; x_2^{d_2}\cdots x_r^{d_r},\; x_1x_2^{d_2-1}x_3^{d_3}\cdots x_r^{d_r},\; \ldots,\; x_1x_2^{d_2}\cdots x_{r-1}^{d_{r-1}}x_r^{d_r-1}] \subset \widehat{R_G}.$$
\end{prop}

\begin{proof}
Note that, according to \eqref{doi},
for any $1\leq j\leq r$, we have:
\begin{equation}\label{doidoi}
u_G/x_j = x_1^{d_1}\cdots x_{j-1}^{d_{j-1}}\cdot x_j^{d_{j}-1}\cdot
   x_{j+1}^{d_{j+1}} \cdots x_{r}^{d_{r}} \in \widehat{R_G}.
\end{equation}
Since $x_1\in R_G \subset \widehat{R_G}$, from \eqref{doidoi} we get the required conclusion.
\end{proof}

A natural question which arise is the following: Is the monoid $M(G)$ normal, for any finite group $G$?
The answer is no, as the following example shows:

\begin{exm}
\emph{
Let $G=A_6$, the alternating group of order $6$. Then:
\begin{align*}
& R_G=K[x_1,\; x_1x_2,\; x_1x_2x_6,\; x_1x_3,\;x_1x_3x_6,\; x_1x_6,\;x_2x_3,\; x_2x_3x_4x_5^2x_6^2x_7^2,\;x_2x_3x_4^2x_5x_6^2x_7^2, \\ 
& \; x_2x_4x_5x_6,\; x_2x_7,\; x_3x_4x_5x_6,\; x_3x_7,\;  x_4x_5x_7^2,\; x_4x_5x_6x_7^2,\; x_7].
\end{align*}
Using Normaliz \cite{normaliz} we compute:
\begin{align*}
& \widehat{R_G}=K[x_1,\; x_1x_2,\; x_1x_2x_6,\; x_1x_3,\;x_1x_3x_6,\; x_1x_6,\;x_2x_3,\; x_2x_3x_4x_5^2x_6^2x_7^2,\;x_2x_3x_4^2x_5x_6^2x_7^2, \\ 
&  \;\underline{x_2x_3x_4x_5x_6},\;  x_2x_4x_5x_6,\; x_2x_7,\; x_3x_4x_5x_6,\; x_3x_7,\;  x_4x_5x_7^2,\; x_4x_5x_6x_7^2,\; x_7].
\end{align*}
Hence, $\widehat{R_G}=R_G[x_2x_3x_4x_5x_6]$ and $R_G\subset \widehat{R_G}$ is strict.
Note that $(x_2x_3x_4x_5x_6)^2\in R_G$. Also, we mention that $G$ is not almost monomial.}
\end{exm}

Let $G$ be a finite group and $N\unlhd G$ a normal subgroup. It is well known that $\Irr(G/N)$ is in bijection with the set 
$$\{\chi\in \Irr(G)\;:\;N\subset \Ker(\chi)\}.$$
For a character $\widetilde \chi \in \Irr(G/N)$,
we denote $\chi$ the corresponding character in $\Irr(G)$, that is $\chi(g):=\widetilde{\chi}(\hat g)$ 
for all $g\in G$, where $\hat g$ is the class of $g$ in $G/N$. 

\begin{teor}
Let $G$ be a finite group and $N\unlhd G$ a normal subgroup. We assume that $\Irr(G)=\{\chi_1,\ldots,\chi_r\}$ and 
$\Irr(G/N)=\{\widetilde{\chi_1},\ldots,\widetilde{\chi_s}\}$, where $s\leq r$. Then:
$$R_{G/N} = R_{G}\cap K[x_1,\ldots,x_s].$$
\end{teor}

\begin{proof}
The inclusion $R_{G/N}\subset K[x_1,\ldots,x_s]$ is clear. 
If $H/N \leqslant G/N$ and $\widetilde \lambda$ is a linear character of $H/N$,
then $\widetilde \lambda^{G/N} = \widetilde{\lambda^G}$. It follows that $R_{G/N}\subset R_G$.

For the converse inclusion, let $u\in R_{G}\cap K[x_1,\ldots,x_s]$. We may assume that there exists
a subgroup $H\leqslant G$ and a linear character $\lambda$ of $H$ such that $u=x_1^{a_1}\cdots x_s^{a_s}$,
where $\lambda^G=a_1\chi_1+\cdots+a_s\chi_s$. Since $N\subset \Ker(\chi_i)$ for any $1\leq i\leq s$,
it follows that $N\subset \Ker(\lambda^G)$ and, therefore, $N\subset H$.
\end{proof}

\begin{teor}\label{p210}
 Let $G_1,G_2$ be two finite groups and assume that $R_{G_1}=K[u_1,\ldots,u_m]\subset K[x_1,\ldots,x_r]$ and
$R_{G_2}=K[v_1,\ldots,v_t]\subset K[y_1,\ldots,y_s]$. Then:
$$R_{G_1\times G_2}=K[u_iv_j:\;:1\leq i\leq m,1\leq j\leq t]\subset K[x_iy_j\;:\;1\leq i\leq m,1\leq j\leq t].$$
In other words, $R_{G_1\times G_2}=R_{G_1}\otimes_K R_{G_2}$.
\end{teor}

\begin{proof}
 It suffices to prove that $\Mon(G_1\times G_2)=\Mon(G_1)\times\Mon(G_2)$. Let $\chi\in\Mon(G_1\times G_2)$. Then
there exist $H\leqslant G$ and $\lambda\in \Lin(H)$ such that $\lambda^G=\chi$. On the other hand $H=H_1\times H_2$
and $\lambda=\lambda_1\times \lambda_2$, where $H_i\leqslant G_i$, $\lambda_i\in\Lin(H_i)$, $i=1,2$. Let $\chi_1=\lambda_1^{G_1}$
and $\chi_2=\lambda_2^{G_2}$. It follows that $\lambda^G=\lambda_1^{G_1}\times \lambda_2^{G_2}$, hence $\Mon(G_1\times G_2)\subset \Mon(G_1)\times\Mon(G_2)$.
The othe inclusion is similar.
\end{proof}

Let $K/\mathbb Q$ be a finite Galois extension with Galois group $G$. Let $s_0\in\mathbb C\setminus \{1\}$ and let 
$$H(s_0)=\{ \chi\in \Ch(G)\;:\;L(s,\chi)\text{ is holomorphic at }s_0\}.$$
The semigroup $H(s_0)$ was firstly introduced by F.\ Nicolae in \cite{monat}.

\begin{teor}\label{p211}
$H(s_0)$ is a normal submonoid of $\Ch(G)$. Moreover, $\widehat{M(G)}\subset H(s_0)$, hence $\widehat{R_G}\subset K[H(s_0)]$.
\end{teor}

\begin{proof}
Since, for any monomial character $\chi$ of $G$, the $L$-function $L(s,\chi)$ is holomorphic on $\mathbb C\setminus\{1\}$, it
follows that $M(G)\subset H(s_0)$.  Let $\chi\in \Ch(G)$ such that 
$d\chi \in H(s_0)$ for some positive integer $d$. It follows that $L(s,d\chi)=L(s,\chi)^d$ is holomorphic at $s_0$, 
hence $L(s,\chi)$ is holomorphic at $s_0$. Therefore, $H(s_0)$ is normal.

Since $M(G)\subset H(s_0)$ and $H(s_0)$ is normal, it follows that $\widehat{M(G)}\subset H(s_0)$, hence $\widehat{R_G}\subset K[H(s_0)]$.
\end{proof}

\begin{cor}
Let $G$ be a finite group. Then:
\begin{enumerate}
 \item[(1)] If $G$ is monomial, then $R_G=\widehat{R_G}=K[H(s_0)]=K[x_1,\ldots,x_r]$. 
 \item[(2)] If $G$ is quasi-monomial, then $K[x_1^{a_1},\ldots,x_r^{a_r}]=R_G \subset \widehat{R_G}=K[H(s_0)]=K[x_1,\ldots,x_r]$.
\end{enumerate}
\end{cor}

\begin{proof}
The proof is immediate from Proposition \ref{p21} and Theorem \ref{p211}.
\end{proof}

\section{Supercharacter theoretic case}

Diaconis and Isaacs \cite{isaacs} introduced the theory of supercharacters as follows:

\begin{dfn}
Let $G$ be a finite group. Let $\mathcal K$ be a partition of $G$ and let $\mathcal X$ be a partition of $\Irr(G)$. 
The ordered pair $C:=(\mathcal X,\mathcal K)$
is a \emph{supercharacter theory} if:
\begin{enumerate}
\item[(1)] $\{1\}\in\mathcal K$,
\item[(2)] $|\mathcal X|=|\mathcal K|$, and
\item[(3)] for each $X\in\mathcal X$, the character $\sigma_X=\sum_{\psi\in \mathcal X}\psi(1)\psi$ is constant on each $K\in\mathcal K$.
\end{enumerate}
The characters $\sigma_X$ are called \emph{supercharacters}, and the elements $K$ in $\mathcal K$ are called \emph{superclasses}.
We denote $\Sup(G)$ the set of supercharacter theories of $G$.
\end{dfn}

Diaconis and Isaacs showed their theory enjoys properties similar to the classical character theory.
For example, every superclass is a union of conjugacy classes in $G$; see \cite[Theorem 2.2]{isaacs}.
The irreducible characters and conjugacy classes of $G$ give a supercharacter 
theory of $G$, which will be referred to as the \emph{classical theory} of $G$.

Also, as noted in \cite{isaacs}, every group $G$ admits
a non-classical theory with only two supercharacters $1_G$ and $\Reg(G)-1_G$ and
superclasses $\{1\}$ and $G\setminus\{1\}$, where $1_G$ denotes the trivial character of $G$
and $\Reg(G) = \sum_{\chi\in\Irr(G)}\chi(1)\chi$ is the regular character of $G$. 
This theory will be called the \emph{maximal theory} of $G$.

Let $C=(\mathcal X,\mathcal K)$ supercharacter theory of $G$ and assume that $\mathcal X=\{X_1,\ldots,X_m\}$.
We let $\Char(G,C):=\{a_1\sigma_{X_1}+\cdots+a_m\sigma_{X_m}\;:\;a_i\in\mathbb Z\},$
the set of \emph{virtual supercharacters} of $G$, w.r.t. $C$. Note that $(\Char(G,C),+)$ is an abelian group,
isomorphic cu $\mathbb Z^p$. Moreover, $\Char(G,C)$ is a subgroup of $\Char(G)$.

Also, we let
$\Ch(G,C):=\{a_1\sigma_{X_1}+\cdots+a_m\sigma_{X_m}\;:\;a_i\in\mathbb N\}\subset \Ch(G).$
We consider the monoid 
$$M(G,C):=\Ch(G,C)\cap M(G).$$
We consider the monomials
$F_i:=\prod_{\chi_j \in X_i}x_j^{d_j},\; 1\leq i\leq m,$ 
where $d_j:=\chi_j(1)$ for all $1\leq j\leq r$. 
Let $$R_{G,C}:=K[M(G,C)] \subset K[F_1,F_2,\cdots,F_m].$$
We also consider the monoid 
$$\widehat{M(G,C)}:=\Ch(G,C)\cap \widehat{M(G)}$$
and the associated $K$-algebra 
$$\widehat{R_{G,C}}:=K[\widehat{M(G,C)}] \subset K[F_1,F_2,\cdots,F_m].$$
We recall several definition from \cite{cimrad}:

\begin{dfn}(\cite[Definition 2.3]{cimrad})\label{cqm-def}
Let $G$ be a finite group and let $C:=(\mathcal X,\mathcal K)$ be a supercharacter theory on $G$. Assume that $\mathcal X=\{X_1,\ldots,X_m\}$.
We say that $G$ is \emph{$C$-quasi monomial} (or a $C-QM$-group), 
if for any $k\in \{1,\ldots,m\}$, there exists some subgroups $H_1,\ldots,H_t \leqslant G$ (not necessarily distinct) and 
linear characters $\lambda_1,\ldots,\lambda_t$ of $H_1,\ldots, H_t$ such that:
$$\lambda_1^G+\cdots+\lambda_t^G=d\sigma_{X_k},$$ 
where $d$ is a positive integer.
\end{dfn}

\begin{prop}\label{p215}
Let $G$ be a finite group and let $C:=(\mathcal X,\mathcal K)$ be a supercharacter theory on $G$. Then, the following are equivalent:
\begin{enumerate}
\item[(1)] $G$ is $C$-quasi monomial.
\item[(2)] $K[F_1^{a_1},\ldots,F_m^{a_m}]\subseteq R_{G,C}$ for some positive integers $a_i$.
\item[(3)] $\widehat{R_{G,C}}=K[F_1,\ldots,F_m]$.
\end{enumerate}
\end{prop}

\begin{proof}
$(1)\Leftrightarrow (2)$ $G$ is $C$-quasi monomial if and only if, for any $1\leq i\leq m$, there exist some subgroups $H_{i1},\ldots,H_{it_i} \leqslant G$ 
and some linear characters $\lambda_{i1},\ldots,\lambda_{it_i}$ of $H_{i1},\ldots,H_{it_i}$ such that
$$\lambda_{i1}^{G}+\cdots+\lambda_{it_i}^G=a_i\sigma_{X_i},$$
fore some positive integer $a_i$, that is  $F_i^{a_i}\in R_{G,C}$.

$(2)\Leftrightarrow (3)$ It is clear, since $R_{G,C}\subset K[F_1,\ldots,F_m]$.
\end{proof}

\begin{dfn}(\cite[Definition 2.9]{cimrad})\label{cam-def}
Let $G$ be a finite group and let $C=(\mathcal X,\mathcal K)$ be a supercharacter theory on $G$. Assume that $\mathcal X=\{X_1,\ldots,X_m\}$.
We say that $G$ is 
\emph{$C$-almost monomial}, 
if for any $k\neq \ell$, there exist some subgroups $H_1,\ldots,H_t \leqslant G$ (not necessarily distinct) and 
linear characters $\lambda_1,\ldots,\lambda_t$ of $H_1,\ldots, H_t$ such that:
$$\lambda_1^G+\cdots+\lambda_t^G=\sum_{i=1}^m \alpha_i\sigma_{X_i},$$ 
where $\alpha_i\geq 0$ are integers with $\alpha_k>0$ and $\alpha_{\ell}=0$.
\end{dfn}

\begin{prop}
Let $G$ be a finite group and let $C=(\mathcal X,\mathcal K)$ be a supercharacter theory on $G$. 
The following are equivalent:
\begin{enumerate}
 \item[(1)] $G$ is $C$-almost monomial.
 \item[(2)] For any $k\neq \ell \in \{1,\ldots,m\}$, there exist some integers $\alpha_i\geq 0$ such that 
            $$F_1^{\alpha_1}\cdots F_{m}^{\alpha_{m}} \in R_{G,C},\text{ with }\alpha_k>0\text{ and }\alpha_{\ell}=0.$$
 \item[(3)] For any $\ell\in\{1,\ldots,m\}$, there exist some integers $\alpha_i>0$ such that 
            $$F_1^{\alpha_1}\cdots F_{\ell-1}^{\alpha_{\ell-1}}\cdot F_{\ell+1}^{\alpha_{\ell+1}} \cdots F_{m}^{\alpha_{m}} \in R_{G,C}.$$
\end{enumerate}
\end{prop}

\begin{proof}
$(1)\Leftrightarrow (2)$ It is immediate from Definition \ref{cam-def}.

$(2)\Leftrightarrow (3)$ It follows from \cite[Proposition 2.11]{cimrad}.
\end{proof}

Let $G$ be a finite group and let $C:=(\mathcal X,\mathcal K)$ be a supercharacter theory of $G$. Let $N\unlhd G$ be a normal subgroup of $G$.
The group $N$ is called $C$-normal or supernormal, if $N$ is a union of superclasses from $C$; see \cite{marberg} and \cite{hend}.
Let $X\in \mathcal X$. By abuse of notation, we write $X\subset \Irr(G/N)$ if for any $\chi \in X$, then $N\subset \Ker(\chi)$.
Let $K\in \mathcal K$. We denote $\widetilde K:=KN/N \subset G/N$.

Now, assume that $N$ is $C$-normal. Without any loss of generality, we can assume that $X_i\subset \Irr(G/N)$ for $1\leq i\leq p$ and $X_i \subsetneq \Irr(G/N)$ for 
$p+1\leq i\leq m$. Let $\widetilde{\mathcal X}:=\{\widetilde{X_1},\ldots,\widetilde{X_p}\}$ and $\widetilde{\mathcal K}:=\{\widetilde K_1,\ldots,\widetilde K_p\}$.
According to \cite[Proposition 6.4]{hend}, the pair
$$\widetilde C:=C^{G/N}=(\widetilde{\mathcal X},\widetilde{\mathcal K})$$
is a supercharacter theory of $G/N$.

\begin{teor}
 With the above notations, we have $R_{G/N,\widetilde C}=R_{G,C}\cap K[F_1,\ldots,F_p]$.
\end{teor}

\begin{proof}
The inclusion $R_{G/N,\widetilde C}\subseteq R_{G,C}\cap K[F_1,\ldots,F_p]$ is clear. Conversely, let $u\in R_{G,C}\cap K[F_1,\ldots,F_p]$.
It follows that $u=F_1^{a_1}\cdots F_p^{a_p}$ for some nonnegative integers $a_i$. Also, there exist some subgroups $H_1,\ldots,H_t$ and 
some linear characthers $\lambda_1,\ldots,\lambda_t$ of $H_1,\ldots,H_t$ such that
$$\lambda_1^{G}+\cdots+\lambda_t^G=a_1\sigma_{X_1}+\cdots+a_p\sigma_{X_p}.$$
Since $X_i\subset \Irr(G/N)$ for all $1\leq i\leq p$, it follows that $N\subset H_j$ for all $1\leq j\leq t$. Hence we are done.
\end{proof}

Let $G$ and $G'$ be two finite groups and let $C=(\mathcal X,\mathcal K)$ and $C'=(\mathcal X',\mathcal K')$ be supercharacter theories of $G$ and $G'$, respectively.
According to \cite[Proposition 8.1]{hend}, $C\times C'$ is a supercharacter theory on $G\times G'$.

\begin{teor}
 With the above notations, we have $R_{G\times G',C\times C'}=R_{G,C}\otimes_K R_{G',C'}.$
\end{teor}

\begin{proof}
 It suffices to show that $M(G\times G',C\times C')=M(G,C)\times M(G',C')$. The proof is similar to the proof of Theorem \ref{p210}.
\end{proof}

Let $K/\mathbb Q$ be a finite Galois extension with Galois group $G$. Let $C:=(\mathcal X,\mathcal K)$ be a supercharacter theory on $G$, as above.
Let $s_0\in\mathbb C\setminus \{1\}$ and let 
$$H(C,s_0)=\{ \chi\in \Ch(G,C)\;:\;L(s,\chi)\text{ is holomorphic at }s_0\}.$$
We then have the following obvious result:

\begin{teor}\label{p220}
$H(C,s_0)$ is a normal submonoid of $\Ch(G,C)$.
Moreoever, we have $M(G,C)\subset \widehat{M(G,C)} \subset H(C,s_0)$, 
hence $R_{G,C}\subset \widehat{R_{G,C}} \subset K[H(C,s_0)]$.
\end{teor}

\begin{cor}
If $G$ is $C$-quasi monomial, then 
$$K[F_1^{a_1},\ldots,F_r^{a_r}] \subset R_{G,C} \subset \widehat{R_{G,C}} = H(C,s_0) = K[F_1,\ldots,F_r].$$
\end{cor}

\begin{proof}
It follows from Proposition \ref{p215} and Theorem \ref{p220}.
\end{proof}

\section{GAP codes and computer experiments}

The following code returns the Hilbert basis of the monoid $M(G)$.

\noindent
$gap>$ LoadPackage("NumericalSgps");

 MonoAlg:=function(g) 

 local cc,i,x,y,z,o,l,j,t,a,af; 

 cc:=ConjugacyClassesSubgroups(g);

 i:=Size(Irr(g));

 t:=[];

 for x in cc do 

 \hspace{10pt}y:=Representative(x); 

\hspace{10pt}for z in LinearCharacters(y) do 

\hspace{20pt}o:=InducedClassFunction(z,g); l:=[];

\hspace{20pt}for j in [1..i] do 

\hspace{30pt}a:= ScalarProduct(Irr(g)[j],o);

\hspace{30pt}Add(l,a);

\hspace{20pt}od;

\hspace{20pt}Add(t,l);

od; od;

af:=AffineSemigroup(t);

return(MinimalGenerators(af));

end;;\\

Our computer experiments yields us to the following conjecture.

\begin{conj}
Let $G:=SL(2,2^n)$, the special linear group of order $2$, over $\mathbb F_{2^n}$. Then:
$$R_G=K[x_1,\;x_1x_{2^{n-1}+2},\;x_{2^{n-1}+3},\cdots,x_{2^n+1},\;u_2,\ldots,u_{2^{n-1}+1},v,w]\subset K[x_1,\ldots,x_{2^n+1}],$$
where $u_k=(x_2x_3\cdots x_{2^n+1})/x_k$ for $2\leq k\leq 2^{n-1}+1$, $v=x_2x_3\cdots x_{2^{n-1}+1}$ and $w=v\cdot x_{2^{n-1}+2}$.
In particular, the group $G$ is almost monomial.
\end{conj}

\begin{obs}\rm
We sketch the proof for the inclusion $\supseteq$:
It is well known \cite{polo} that $G:=SL(2,2^n)$ has $q=2^n+1$ irreducible characters: $\chi_1$, the trivial character, $\chi_2,\ldots,\chi_{2^{n-1}+1}$, characters of order $q-1$, $\chi_{2^{n-1}+2}$, a character of order $q$, and $\chi_{2^{n-1}+3},\ldots,\chi_{2^n+1}$, characters of order $q+1$. 
The character $\chi_1$ is obviously monomial. Also, the characters $\chi_{2^{n-1}+3},\ldots,\chi_{2^n+1}$ and $\chi_1\chi_{2^{n-1}+2}$ are induced by the
linear characters of a subgroup $H$ of $G$ of order $q(q-1)$.

Let $U_k:=\chi_2+\cdots+\widehat{\chi_k}+ \cdots+\chi_{2^n+1}$, $2\leq k\leq 2^{n-1}+1$. Note that $U_k(1)=q(q-1)$ for any $2\leq k\leq 2^{n-1}+1$. The characters $U_k$
are induced by linear characters of a subgroup of order $q+1$. Similarly, the characters $V:=\chi_2+\chi_3+\cdots+\chi_{2^{n-1}+1}$ and $W:=V+\chi_{2^{n-1}+2}$, are
induced by linear characters of some subgroups of order $2(q+1)$ and $2(q-1)$ respectively.
\end{obs}

\textbf{Acknowledgment.} We gratefully acknowledge the use of the computer algebra system GAP (\cite{gap}) and Normaliz (\cite{normaliz}) for our experiments. This first author was supported by a grant of the Ministry of Research, Innovation and Digitization, CNCS - UEFISCDI, project number PN-III-P1-1.1-TE-2021-1633, within PNCDI III.

{}

\vspace{2mm} \noindent {\footnotesize
\begin{minipage}[b]{15cm}
\emph{Mircea Cimpoea\c s}, University Politehnica of Bucharest, Faculty of
Applied Sciences, 
Bucharest, 060042, Romania and Simion Stoilow Institute of Mathematics, Research unit 5, P.O.Box 1-764,
Bucharest 014700, Romania  \\ 
 E-mail: mircea.cimpoeas@upb.ro,\;mircea.cimpoeas@imar.ro
\end{minipage}}

\vspace{2mm} \noindent {\footnotesize
\begin{minipage}[b]{15cm}
\emph{Alexandru Florin Radu}, University Politehnica of Bucharest, Faculty of
Applied Sciences, 
Bucharest, 060042, Romania.\\ 
 E-mail: sasharadu@icloud.com
\end{minipage}}


\begin{thebibliography}{}
\bibitem{arama} H.\ Aramata, \emph{Uber die Teilbarkeit der Dedekindschen Zetafunktionen}, Proc. Imp. Acad. Tokyo \textbf{7} (1931), 334--336.

 \bibitem{artin2} E.\ Artin, \emph{Zur Theorie der L-Reihen mit allgemeinen Gruppencharakteren}, Abh. Math. Sem. Hamburg \textbf{8} (1931), 292--306.

\bibitem{brauer} R.\ Brauer, \emph{On Artin’s L-series with general group characters}, Ann. of Math. \textbf{48} (1947), 502--514.

\bibitem{normaliz} W. Bruns, B. Ichim, C. Söger and U. von der Ohe: Normaliz. Algorithms for rational cones and affine monoids. Available at https://www.normaliz.uni-osnabrueck.de.




\bibitem{cimrad} M.\ Cimpoea\c s, A.\ F.\ Radu, \emph{On supercharacter theoretic generalizations of monomial groups and Artin's conjecture},
        to be published in Czechoslovak Mathematical Journal.


\bibitem{isaacs} P.\ Diaconis, I.\ M.\ Isaacs, \emph{Supercharacters and superclasses for algebra groups},
Trans. Amer. Math. Soc. \textbf{360} (2008), 2359--2392.



 \bibitem{gap}
  The GAP~Group, \emph{GAP -- Groups, Algorithms, and Programming, 
  Version 4.11.1}; 
  2021, \verb+(https://www.gap-system.org)+.

\bibitem{hend} A.\ O.\ F.\ Hendrickson, \emph{Supercharacter theory constructions corresponding to Schur ring products}, Comm. Algebra
\textbf{40, no. 12} (2012), 4420--4438.


\bibitem{cc} A.\ V.\ Lopez, J.\ V.\ Lopez, \emph{Classification of finite groups according to the number of conjugacy classes},
\bibitem{marberg} E.\ Marberg, \emph{A supercharacter analogue for normality}, J. Algebra \textbf{332} (2011), 334--365.

\bibitem{polo} G.\ S.\ Narkunskaja, A.\ V.\ Zevelinksi, \emph{Representation of the group $SL(n,\mathbb F_q)$, where $q=2^n$}, Functional. Anal. i. Prilo\u zen,
               8(3) (1974), 75--76.


\bibitem{monat} F.\ Nicolae, \emph{On holomorphic Artin L-functions}, Monatsh. Math. \textbf{186, no. 4}, (2018), 679--683.

\bibitem{rhoades} S.\ L.\ Rhoades, \emph{A generalization of the Aramata-Brauer Theorem}, Proceedings of the
American Mathematical Society \textbf{Vol. 119, No. 2} (1993), 357--364.


\end{thebibliography}
\end{document}